\documentclass[12pt]{amsart}
\usepackage[active]{srcltx}
\usepackage{a4wide}
\usepackage{amsthm,amsfonts,amsmath,mathrsfs,amssymb}
\usepackage{dsfont}
\usepackage[T1]{fontenc}
\usepackage[utf8]{inputenc}
\usepackage{fourier}
\usepackage{pdfpages}
\usepackage{graphicx}

\def\a{\alpha}               \def\g{\gamma}
            \def\om{\omega}
       \def\t{\theta}       
         \def\r{\rho}         \def\z{\zeta}

\def\e{\varepsilon}

\def\D{{\mathbb D}}

\def\({\left(}       \def\){\right)}

\newtheorem{prop}{\sc Proposition}

\newtheorem{thm}[prop]{\sc Theorem}

\newtheorem{ex}[prop]{\sc Example}

\begin{document}
\title[Korenblum's principle for Bergman spaces with radial weights]{Korenblum's principle for Bergman spaces with radial weights}
\author[I. Efraimidis]{Iason Efraimidis}
\address{Departamento de Matem\'aticas, Universidad Aut\'onoma de
Madrid, 28049 Madrid, Spain}
\email{iason.efraimidis@uam.es}
\author[A. Llinares]{Adri\'an Llinares}
\address{Department of Mathematics and Mathematical Statistics,
Ume{\char229} University, 90736 Ume{\char229}, Sweden. Current address: Departamento de An\'alisis Matem\'atico y Matem\'atica Aplicada, Universidad Complutense de Madrid, 28040 Madrid, Spain}
\email{adrialli@ucm.es}
\author[D. Vukoti\'c]{Dragan Vukoti\'c}
\address{Departamento de Matem\'aticas, Universidad Aut\'onoma de
Madrid, 28049 Madrid, Spain} \email{dragan.vukotic@uam.es}
\dedicatory{Peter Duren (1935-2020) in memoriam}
\thanks{All authors are partially supported by PID2019-106870GB-I00 from MICINN, Spain. The first author is supported by a María Zambrano contract, reference number CA3/RSUE/2021-00386, from UAM and Ministerio de Universidades, Spain (Plan de Recuperación, Transformación y Resiliencia). Second author's work is funded by the postdoctoral scholarship JCK22-0052 granted by The Kempe Foundations}
\subjclass[2010]{}
\date{05 April, 2024.}

\begin{abstract}
We show that the Korenblum maximum (domination) principle is valid for weighted Bergman spaces $A^p_w$ with arbitrary (non-negative and integrable) radial weights $w$ in the case $1\le p<\infty$. We also notice that in every weighted Bergman space the supremum of all radii for which the principle holds is strictly smaller than one. Under the mild additional assumption $\liminf_{r\to 0^+} w(r)>0$, we show that the principle fails whenever $0<p<1$.
\end{abstract}
\maketitle
\section{Introduction}
 \label{sect-intro}
\par
\subsection{Korenblum domination (maximum) principle}
\label{subsect-K-princ}
Let $A^p$ denote the Bergman spaces of all analytic functions in the unit disk $\D$ that are $p$-integrable with respect to the normalized Lebesgue area measure $dA$. In the late 1980s, Korenblum conjectured that the following principle should be valid: there exists a radius $c\in (0,1)$ such that, for any two functions $f$, $g\in A^2$ with $|f(z)| \le |g(z)|$ in $A_c=\{z\,:\,c\le |z|<1\}$, we have $\|f\|_2\le \|g\|_2$. He proved the statement in \cite{Ko} under some additional assumptions on the zeros of $f$ and $g$. It is quite easy to formulate a  generalization of this question for any positive $p$. Various authors produced related partial results on such questions during the 1990s.
\par
In 1999, Hayman \cite{Ha} proved this conjecture for $p=2$, showing that $c=1/25=0.04$ works (and hence so does any smaller positive value). Shortly afterwards, Hinkkanen \cite{Hi}  extended the result for any $p\ge 1$ and $c=0.15724$. Their methods of proof were somewhat similar. Hinkkanen relies on one fixed estimate and therefore proved the result for the same radius, independently of the value of $p$. Further relevant technical improvements were obtained by  Schuster \cite{S} and also by Wang \cite{W1, W2, W3, W4} regarding the largest possible radius $c$.
\par
In 2018, Bo\v zin and Karapetrovi\'c \cite{BK} showed that Korenblum's principle fails for all $p$ with $0<p<1$ by exhibiting a universal counterexample. More recently, Karapetrovi\'c \cite{Ka} studied the problem for the more general mixed-norm spaces $H(p,q,\a)$ which include the Bergman spaces with standard radial weights as special cases when $p=q$. He was able to show that the principle holds when $1\le p \le q < \infty$ and fails when $0<q<1$.
\par
\subsection{Description of our results}
 \label{subsect-descr}
Karapetrovi\'c's carefully written paper \cite{Ka}, together with other earlier works, inspired our research and made us think of generalizing the results to Bergman spaces with general radial weights. In fact, it turns out that the previous results, both the positive and the negative ones, carry over with quite similar proofs but also with some further simplifications to these general cases.
\par
In the present note we show that the Korenblum principle actually holds in any weighted Bergman spaces $A^p_w$ with  a non-negative and integrable radial weight $w$ in the case $p\ge 1$ (Theorem~\ref{thm-large-p}). We also notice than in general, if the principle holds, the supremum of all admissible radii is strictly smaller than one (Proposition~\ref{prop-sup1}). Under the additional assumption $\liminf_{r\to 0^+} w(r)>0$, we show that the principle fails whenever $0<p<1$ (Theorem~\ref{thm-small-p}) and if such an assumption is omitted, it may hold even then (Example~\ref{ex-w-0}).
\par
We notice that proving these facts actually does not require changing many details from the existing discussions in the literature, something that seems to have remained unnoticed until now. All the necessary basic concepts are reviewed and detailed proofs are given in the next section.

\section{Domination principle for weighted Bergman spaces with radial weights}
 \label{sect-rad-weights}
\par
\subsection{Hardy and Bergman spaces}
\label{subsect-Hp-Ap}
It is well known \cite[Chapter~1]{D} that for a subharmonic function $g$ in the unit disk $\D$ its integral means
$$
 \int_0^{2\pi} g(re^{i\t}) \frac{d\t}{2\pi}\,, 
$$
on the circle $\{z\,:\,|z|=r\}$, $0<r<1$, are increasing functions of the radius $r$. In particular, for any  function $f$ analytic in $\D$ and $0<p<\infty$, the integral means of order $p$ of $f$:
$$
 M_p(r;f) = \( \int_0^{2\pi} |f(re^{i\t})|^p \frac{d\t}{2\pi} \)^{1/p}
$$
are increasing functions of $r\in (0,1)$. The Hardy space $H^p$ is the set of all analytic functions in $\D$ for which these means have finite limits: $\|f\|_{H^p}=\lim_{r\to 1^-}M_p(r;f)<\infty$. This is not a true norm if $0<p<1$ but the same notation is still used in this case. \par
Let $dA$ denote the normalized Lebesgue area measure: $dA = \frac{1}{\pi}\,dx\,dy = \frac{1}{\pi}\,r\,dr\,d\t$, where $z=x+iy=re^{i\t}\in\D$. For $p\in (0,\infty)$, the Bergman space $A^p$ is defined as the set of all analytic functions $f$ in $\D$ that are $p$-integrable with respect to area measure:
$$
 \|f\|_p = \( \int_\D |f|^p\,dA \)^{1/p} = \( \int_0^1 2r M_p^p(r;f)\,dr \)^{1/p} < \infty\,.
$$
Again, the same comment regarding the case $0<p<1$ as above is in place. For the theory of Bergman spaces, we refer the reader to \cite{DS} or \cite{HKZ}. 
\par
\subsection{Weighted Bergman spaces}
 \label{subsect-w-Bergman}
In this section, the term \textit{weight\/} will be used for a measurable (with respect to $dA$) and non-negative function $w$ on $\D$. (Note that non-zero constant functions belong to the space if and only if $w$ is Lebesgue integrable.) To avoid trivial situations, we also assume that $\int_\D w\,dA>0$; equivalently, $w$ is positive on some subset of positive measure of the disk. For $0<p<\infty$, the weighted Bergman space $A^p_w$ is the set of all holomorphic functions $f$ in $\D$ for which
$$
 \|f\|_{p,w} = \( \int_\D |f|^p w\,dA \)^{1/p} < \infty\,,
$$
\par
We will say that the \textit{Korenblum domination principle holds for\/} $A^p_w$ if there exists $c\in (0,1)$ such that for arbitrary functions $f$, $g\in A^p_w$ that satisfy $|f(z)|\le |g(z)|$ whenever $c\le |z|<1$, the inequality $\|f\|_{p,w}\le \|g\|_{p,w}$ holds. Any such value $c$ is said to be an \textit{admissible radius\/}. If no such $c$ exists, we will say that the \textit{Korenblum principle fails for\/} $A^p_w$.
\par
Trivially, if $c$ is an admissible radius, any smaller positive value is also admissible; hence, the set of all admissible radii is an interval. The supremum of all admissible values is usually called the \textit{Korenblum radius\/} and is sometimes denoted by $c(p)$.
\par
The exact value of $c(p)$ is not known for any $p\in (0,\infty)$, even for unweighted Bergman spaces. Wang has improved various upper and lower bounds from earlier papers in \cite{W1, W2, W3} and has also shown in \cite{W4} that $\lim_{p\to\infty} c(p)=1$, thus answering a question posed in \cite{Hi} by Hinkkanen. It is not known whether $c(p)$ is an increasing or continuous function of $p$. Such questions are certainly of interest but will not be considered in the present note.
\par
\subsection{The Korenblum radius is smaller than one for weighted Bergman spaces}
 \label{subsect-Kr<1}
Considering the Korenblum principle and Korenblum radius also makes  sense for other spaces, such as Hardy spaces or the more general mixed norm spaces. Every function $f\in H^p$ has radial limits almost everywhere and $\|f\|_{H^p}$ can be computed in terms of the modulus of its radial limits only. From here it is immediate that every $c\in (0,1)$ is an admissible radius. Therefore the Korenblum radius is equal to $1$ for all $H^p$ spaces.
\par
Our next result shows that in weighted Bergman spaces this situation is impossible. Here we will still not assume the integrability of the weight but we will consider only the cases when the space $A^p_w$ is non-trivial.
\begin{prop}\label{prop-sup1}
Let $0<p<\infty$ and let $w$ be a weight (not necessarily integrable)  and assume that $A^p_w\neq \{0\}$. If the Korenblum principle holds for such space $A^p_w$, then the corresponding Korenblum radius must satisfy $c(p)<1$.
\end{prop}
\begin{proof}
Assume that, on the contrary, there exists an increasing sequence of admissible radii $(c_n)_n$ such that $\lim_{n\to\infty}c_n=1$. Given $f\in A^p_w\setminus\{0\}$, consider the functions $f_n$ given by
$$
 f_n(z) = \frac{z}{c_n} f(z) \,.
$$
Clearly, $f_n\in A^p_w\setminus\{0\}$ and $|f(z)|\le |f_n(z)|$ whenever $c_n\le |z|<1$. Since $c_n$ is an admissible radius by assumption, it follows that
$$
 \int_\D |f(z)|^p w(z)\,dA(z) \le \frac{1}{c_n^p} \int_\D |z|^p\, |f(z)|^p w(z)\,dA(z) \,.
$$
After taking the limit as $n\to\infty$, this implies that
$$
 \int_\D (1 - |z|^p) |f(z)|^p w(z)\,dA(z) \le 0 \,,
$$
which is only possible if $|f(z)|^p w(z)=0$ almost everywhere in $\D$. Since  $w(z)=0$ cannot be satisfied almost everywhere in $\D$ in view of our assumption on $w$, it follows that $f(z)=0$ on a set of positive measure in $\D$, hence $f\equiv 0$, which is absurd. This proves the claim.
\end{proof}
\par
\subsection{Radial weights}
 \label{subsect-rad-w}
From now on, we only consider integrable (non-negative) weights; hence all bounded analytic functions in $\D$ will belong to $A^p_w$.
\par
A weight is said to be \textit{radial\/} if $w(z)=w(|z|)$ for all $z\in\D$. Our earlier assumptions in this case include the condition that $w(r)>0$ for all $r$ in some subset of positive measure of $(0,1)$. In this case, the norm of a function $f\in A^p_w$ can be computed as follows:
$$
 \|f\|_{p,w} = \( \int_0^1 2r w(r) M_p^p(r;f)\,dr \)^{1/p}\,,
$$
where $M_p(r;f)$ is as before.
\par
It should be noted that $A^p_w$ may not be complete for some weights of this type (for example, if $w(r)=0$ on $(R,1)$ for some $R\in (0,1)$, as the reader can check after a certain amount of work, but this will not be relevant at all for our purpose. In other words, our discussion is independent of the completeness of the space.
\par
\subsection{The case $p\ge 1$}
 \label{subsect-p>=1}
Our first result is positive, generalizing the earlier findings for $p\ge 1$. We follow the methods of \cite{Hi} and \cite{S}, explained very well in \cite{Ka}. Notice that in our considerations we only need the simplest limit of a certain integral, without any need to compute or estimate certain values, unlike in earlier papers. This is achieved  at the expense of not giving any explicit bounds on the value of $c$ for which the principle holds.
\par
\begin{thm}\label{thm-large-p}
Let $1\le p<\infty$ and let $w$ be a radial weight. Then the Korenblum domination principle holds for the space $A^p_w$.
\end{thm}
\begin{proof}
Let $0<c<1$ (the appropriate value is to be defined later) and consider $f$, $g\in A^p_w$ that satisfy $|f(z)|\le |g(z)|$ whenever $c\le |z|<1$.
\par
Define $\om (z) = \frac{f(z)}{g(z)}$ for $c<|z|<1$. Note that $\om$ is holomorphic in this annulus and without loss of generality we may assume that $|\om|<1$ there; otherwise $f$ and $g$ coincide and there is nothing left to prove. Also, if $\om$ is a constant, then $|f|$ is a constant multiple of $|g|$ with the constant $<1$, and the statement of the Korenblum principle follows trivially; thus, we may assume from now on that $\om$ is not constant.
\par
If it is that case that 
$$
 \int_{ \{|z|<c\} } \(|f|^p - |g|^p\) w\,dA \le 0\,,
$$
then after integrating the positive function $\(|g|^p - |f|^p\) w$ over the annulus $\{c<|z|<1\}$, simple algebra readily yields the inequality $\|f\|_{p,w}^p\le \|g\|_{p,w}^p$, and the statement of the Korenblum principle follows. Thus, from now on we may restrict our attention only to the case when
$$
 \int_{ \{|z|<c\} } \(|f|^p - |g|^p\) w\,dA > 0\,.
$$
For each $\r$ in the interval $(c,1)$ there exists a point $\z_\r$ such that
$$
 |\z_\r|=\r\,, \quad |\om(\z_\r)| = \max \{|\om (z)|\,:\,|z|=\r\}\,.
$$
Next, define
$$
 \g(\r) = \max \left\{\frac{|\om (z)-\om(\z_\r)|}{1-|\om (z)|^2}\,:\,|z|=\r\right\} \,.
$$
Note that $\g(\r)>0$ for all $\r \in (c,1)$ since $\om$ is not constant. It is important to stress that $\g$ depends on the choice of functions $f$, $g$ that satisfy our assumptions.
\par
We continue following Hinkkanen \cite{Hi} and using some basic inequalities. To this end, in the next step the assumption $p\ge 1$ is actually fundamental since we will use the elementary double inequality
\begin{equation}
 p y^{p-1} (x-y) \le x^p - y^p \le p x^{p-1} (x-y)\,, \quad x\,, y\ge 0\,, \quad p\ge 1\,.
 \label{eq-elem-ineq}
\end{equation}
\par
For $0<r<c<\r<1$, using the fact that $|\om(\z_\r)|<1$ for all $z$ with $|z|=\r$ and the second inequality in \eqref{eq-elem-ineq}, combined with the basic triangle inequality, we obtain the following chain of estimates
\begin{eqnarray*}
 M_p^p(r, f) - M_p^p (r, g) & = & \int_0^{2\pi} \( |f(r e^{i \t})|^p - |g(r e^{i \t})|^p \) \frac{d\t}{2\pi}
\\
 &\le & \int_0^{2\pi} \( |f(r e^{i \t})|^p - |\om(\z_\r) g(r e^{i \t})|^p \) \frac{d\t}{2\pi}
\\
 &\le & \int_0^{2\pi} p |f(r e^{i \t})|^{p-1} \( |f(r e^{i \t})| - |\om(\z_\r) g(r e^{i \t})| \) \frac{d\t}{2\pi}
\\
 &\le & \int_0^{2\pi} p |f(r e^{i \t})|^{p-1} \left|f(r e^{i \t}) - \om(\z_\r) g(r e^{i \t})\right| \frac{d\t}{2\pi}
\\
 &\le & \int_0^{2\pi} p |f(\r e^{i \t})|^{p-1} \left|f(\r e^{i \t}) - \om(\z_\r) g(\r e^{i \t})\right| \frac{d\t}{2\pi}\,,
\end{eqnarray*}
where in the last step we have used the fact that the function $|f|^{p-1} |f - \om(\z_\r) g|$ is subharmonic. (To check this, just note that its logarithm is clearly subharmonic since $p-1\ge 0$ and recall that the exponential function of a subharmonic function is subharmonic.)
\par
We continue the above chain of inequalities by using the definitions of $\om(z)$ and $\g(\r)$ and the first inequality in \eqref{eq-elem-ineq}:
\begin{eqnarray*}
 M_p^p(r, f) - M_p^p (r, g) &\le & \int_0^{2\pi} p |f(\r e^{i \t})|^{p-1} \( |g(\r e^{i \t})| - |f(\r e^{i \t})| \) \frac{\left|f(\r e^{i \t}) - \om(\z_\r) g(\r e^{i \t})\right|}{|g(\r e^{i \t})| - |f(\r e^{i \t})|} \frac{d\t}{2\pi}
\\
 & = & \int_0^{2\pi} p |f(\r e^{i \t})|^{p-1} \( |g(\r e^{i \t})| - |f(\r e^{i \t})| \) \frac{|\om (\r e^{i \t})-\om(\z_\r)|}{1-|\om (\r e^{i \t})|} \frac{d\t}{2\pi}
\\
 & \le & 2 \g (\r) \int_0^{2\pi} \( |g(\r e^{i \t})|^p - |f(\r e^{i \t})|^p \) \frac{d\t}{2\pi}
\\
 & = & 2 \g (\r) \(M_p^p(\r,g) - M_p^p (\r,f) \)\,.
\end{eqnarray*}
For a fixed $\r$, integrate from $0$ to $c$ with respect to $2r w(r)\,dr$:
$$
 \int_0^c \(M_p^p(r, f) - M_p^p (r, g)\)\,2r w(r)\,dr \le 2 \g (\r) \(M_p^p(\r, g) - M_p^p (\r, f)\)\, \int_0^c 2 r w(r)\,dr\,,
$$
Next, divide both sides by $\g(\r)$ and then integrate from $c$ to $1$ with respect to $\r w(\r)\,d\r$ to get
\begin{equation}
 \int_c^1 \frac{\r w(\r)}{\g (\r)}\,d\r \cdot \int_{ \{|z|<c\} } \(|f|^p - |g|^p\) w\,dA \le \int_0^c 2\,r\,w(r)\,dr \cdot \int_{ \{c<|z|<1\} } \(|g|^p - |f|^p\) w\,dA\,.
 \label{eq-key-ineq}
\end{equation}
In view of our assumption that the second factor on the left is positive, it follows that
$$
 \int_c^1 \frac{\r w(\r)}{\g (\r)}\,d\r <\infty\,.
$$
If we can find a value $c\in (0,1)$ independent of $\g$ (that is, independent of $f$ and $g$) and such that
\begin{equation}
 \int_0^c 2\,r\,w(r)\,dr \le \int_c^1 \frac{\r w(\r)}{\g (\r)}\,d\r \,,
 \label{eq-ineq}
\end{equation}
then, after a cancelation, \eqref{eq-key-ineq} will immediately yield
$$
 \|f\|_{p,w}^p = \int_{ \{|z|<c\} } |f|^p w\,dA  + \int_{ \{c<|z|<1\} } |f|^p w\,dA \le \int_{ \{|z|<c\} } |g|^p w\,dA + \int_{ \{c<|z|<1\} } |g|^p w\,dA = \|g\|_{p,w}^p\,,
$$
which will prove our claim.
\par
To complete the proof, it remains to show that \eqref{eq-ineq} holds for some $c\in (0,1)$ independent of $\g$. On the one hand, the Monotone Convergence Theorem implies that
$$
 \lim_{c\to 0^+} \int_0^c 2\,r\,w(r)\,dr = 0\,.
$$
On the other hand, if $0<c<\frac{1}{4}$ we can use fundamental  estimates due to Schuster. Specifically, in \cite[p.~3528]{S} the following function was defined:
$$
 F(\r,c) = \frac{2c}{\r} \( 1 + \frac{\r^2}{c}\) \cdot \frac{1-c^{12}}{1-c^{10}} \cdot \prod_{n=1}^5 \frac{(1+\r^2 c^{2n-1}) (1+\r^{-2} c^{2n+1}) (1+ c^{2n})^2}{(1+\r^2 c^{2n-2}) (1+\r^{-2} c^{2n}) (1+ c^{2n-1})^2}
$$
and it was shown that
$$
 \g(\r) \le \frac{F(\r,c)}{ \sqrt{1 - F(\r,c)^2} } = H(\r,c)\,.
$$
Note that $H$ is an explicit function independent of $f$ and $g$ (hence of $\g$). It also has the property that
$$
 \lim_{c\to 0^+} H(\r,c) = \frac{2\r}{1-\r^2}\,.
$$
Recalling that $w>0$ on a set of positive measure, Fatou's Lemma then implies that
$$
 0 < \frac12 \int_0^1 w(\r) (1 - \r^2)\,d\r \le \liminf_{c\to 0^+} \int_c^1 \frac{\r w(\r)}{H(\r,c)}\,d\r \le \liminf_{c\to 0^+} \int_c^1 \frac{\r w(\r)}{\g(\r)}\,d\r\,.
$$
Thus, there exists a small enough $c>0$ such that \eqref{eq-ineq} holds independently of $\g$.
\par
This shows that the Korenblum principle holds for such a radius $c$ and $A^p_w$.
\end{proof}
\par
\subsection{The case $0<p<1$}
 \label{subsect-p<1}
Next, we show that under a mild additional assumption on $w$, the Korenblum principle fails for $A^p_w$ whenever $0<p<1$. It turns out that the same example from \cite{Ka} works, with appropriate modifications of the argument employed there (and with some minor simplifications made in passing).
\par
\begin{thm}\label{thm-small-p}
Let $0<p<1$ and let $w$ be a radial weight such that $\,\liminf_{r\to 0^+} w(r)>0$. Then the Korenblum domination principle fails for the space $A^p_w$.
\end{thm}
\begin{proof}
Let $c\in (0,1)$ be arbitrary and choose $n$ large enough so that $n(1-p)>2$. Consider the functions
$$
 f(z) = \frac{c^n}{c^n+\e^n} \(z^n+\e^n\)\,, \quad g(z)=z^n\,,
$$
where an appropriate small value of $\varepsilon\in (0,c)$ is to be chosen later. Note that if $c\le |z|<1$ then clearly
$$
 |f(z)| \le \frac{c^n}{c^n+\e^n} \(|z|^n+\e^n\) \le \frac{c^n |z|^n+|z|^n\e^n}{c^n+\e^n} = |g(z)|\,.
$$
We need to show that $\|g\|_{p,w} <  \|f\|_{p,w}$; that is,
\begin{equation}
 \frac{\(c^n+\e^n\)^p}{c^{np}} \int_0^1 2r w(r) r^{np}\,dr < \int_0^1  2r w(r) M_p^p (r;z^n+\e^n)\,dr\,.
 \label{ineq-norm}
\end{equation}
In order to prove this, define
$$
 I = \int_0^\e 2 r w(r) M_p^p (r;z^n+\e^n)\,dr\,, \quad J = \int_\e^1 2 r w(r) M_p^p (r;z^n+\e^n)\,dr\,.
$$
By subharmonicity of the function $u(z)=|z^n+\e^n|^p$ and the sub-mean value inequality, we have $M_p^p (r;z^n+\e^n) \ge \e^{np}$ for all $r\in (0,1)$, hence
\begin{equation}
 I \ge \e^{np} \int_0^\e 2 r w(r)\,dr  = I^\prime\,.
 \label{ineq-I}
\end{equation}
As was observed in \cite{Ka}, the inequality is actually strict; this can be seen quickly by resorting to the Hardy-Stein identity. However, this is not really needed for our purpose because in a chain of inequalities to be obtained later there will be another strict inequality.
\par
One can also see, as in \cite[p.~501]{Ka}, that $M_p^p (r;z^n+\e^n) \ge r^{np}$, for $\e<r<1$. Again, strict inequality can be proved but is not needed. Actually, this inequality follows from the one observed earlier: $M_p^p (r;z^n+\e^n) \ge \e^{np}$. This is easily seen as follows (by an obvious change of variable $\t \mapsto -\t$):
$$
 M_p^p (r;z^n+\e^n) = \int_{-\pi}^{\pi} |r^n e^{i n \t}+\e^n|^p \frac{d\t}{2\pi} = \int_{-\pi}^{\pi} |r^n +\e^n e^{- i n \t}|^p \frac{d\t}{2\pi} = \int_{-\pi}^{\pi} |r^n +\e^n e^{i n \t}|^p \frac{d\t}{2\pi} \ge r^{n p}
$$
as before. Thus,
\begin{equation}
 J \ge \int_\e^1 2 r w(r) r^{np} \,dr = J^\prime\,,
 \label{ineq-J}
\end{equation}
In view of \eqref{ineq-I} and \eqref{ineq-J}, by Bernoulli's elementary inequality:
$$
 \frac{\(c^n+\e^n\)^p}{c^{np}} = \( 1 + \(\frac{\e}{c}\)^n\)^p \le 1 + p \(\frac{\e}{c}\)^n\,, \quad 0<p<1\,,
$$
we see that, in order to obtain inequality \eqref{ineq-norm}, it suffices to show that
\begin{equation}
 I^\prime + J^\prime > \( 1 + p \(\frac{\e}{c}\)^n \) \int_0^1 2 r w(r) r^{np}\,dr\,.
 \label{ineq-stronger}
\end{equation}
The last inequality is equivalent to
$$
 \e^{np} \int_0^\e 2 r w(r) \,dr + \int_\e^1 2 r w(r) r^{np} \,dr > \( 1 + p \(\frac{\e}{c}\)^n \) \int_0^1 2 r w(r) r^{np}\,dr\,,
$$
which is the same as
\begin{equation}
 \e^{np} \int_0^\e 2 r w(r)  \,dr - \int_0^\e 2 r w(r) r^{np} \,dr > p \(\frac{\e}{c}\)^n \int_0^1 2 r w(r) r^{np}\,dr\,.
 \label{ineq-equiv}
\end{equation}
By putting together the two integrals on the left-hand side of \eqref{ineq-equiv} and employing the obvious change of variable $r=\e \r$, we see that \eqref{ineq-equiv} amounts to
$$
 \e^{np+2} \int_0^1 (1 - \r^{np}) 2 \r w(\e \r) \,d\r > p \(\frac{\e}{c}\)^n \int_0^1 2 r w(r) r^{np}\,dr\,.
$$
Changing back $\r$ to $r$, this shows that the desired inequality \eqref{ineq-stronger} can be rewritten as
\begin{equation}
 \int_0^1 (1 - r^{np}) 2 r w(\e r) \,dr > \frac{p}{c^n} \e^{n (1-p)-2} \int_0^1 2 r w(r) r^{np}\,dr\,.
 \label{ineq-final}
\end{equation}
In view of our initial assumption on $n$, the right-hand side of \eqref{ineq-final} tends to zero as $\e\to 0^+$. Without using any assumption on the continuity of $w$, one can show by elementary calculus that
$$
 \liminf_{\e\to 0^+} w(\e r) = \liminf_{s\to 0^+} w(s)
$$
for every fixed $r\in (0,1)$. (Just consider the appropriate sequences to see that each quantity is greater than or equal to the other.) Thus, taking the $\liminf_{\e\to 0^+}$ of the left-hand side of \eqref{ineq-final}, by Fatou's lemma we obtain that
$$
 \liminf_{\e\to 0^+} \int_0^1 (1 - r^{np}) 2 r w(\e r) \,dr \ge \liminf_{s\to 0^+} w(s) \int_0^1 (1 - r^{np}) 2 r \,dr = \frac{np}{np+2} \liminf_{s\to 0^+} w(s) >0\,.
$$
This shows that we can choose a small enough positive $\e$ so that \eqref{ineq-final} and therefore \eqref{ineq-stronger} holds, which proves \eqref{ineq-norm}.
\end{proof}
Finally, it should be noted that, in order for the Korenblum principle to fail for $0<p<1$, an assumption like $\,\liminf_{r\to 0^+} w(r)>0$ on the weight cannot be dispensed with, as the following simple example shows.
\begin{ex}\label{ex-w-0}
Let
$$
 w(r) =
 \left\{
  \begin{array}{ll}
    0 & \mbox{if \ $0\le r<R$}, \\
    1 & \mbox{if \ $R\le r<1$}\,.
  \end{array}
 \right.
$$
Then $w$ clearly does not satisfy the condition $\,\liminf_{r\to 0^+} w(r)>0$, yet the Korenblum principle trivially holds for $A^p_w$ for every $p>0$ with $c=R$ since
$$
 \|f\|_{p,w} = \( \int_R^1 2r w(r) M_p^p(r;f)\,dr \)^{1/p} \,.
$$
\end{ex}
\par
\subsection{Conclusions}
 \label{subsect-concl}
It should be pointed out that the main credit for our findings belongs to the previous authors: to Hayman \cite{Ha} and Hinkkanen \cite{Hi} for devising the method used in the proof of Theorem~\ref{thm-large-p} and to Schuster \cite{S} for making clear the estimates used for all values of $p$, as well as to Bo\v zin and Karapetrovi\'c's \cite{BK} for producing a single and simple counterexample that works for all $p\in (0,1)$, even in the general context of our Theorem~\ref{thm-small-p}.  Also, Karapetrovi\'c very clear exposition in \cite{Ka} motivated this work.
\par
Our contribution consists primarily in realizing that the estimates obtained by the other authors actually work in the context of general Bergman spaces with radial weights in place of the spaces with standard radial weights, and also in noticing that all the estimates work without having to estimate any values of the Beta function that come  up in the computation for the standard cases.
\par
Our findings seem to show that the Korenblum principle is a purely classical analysis statement that only requires standard estimates on integral means, without resorting to any abstract context or ``soft analysis'' methods, and its validity or failure are independent of the question of completeness of the weighted Bergman space considered, which is another remarkable fact.


\end{document}